\documentclass[a4paper,11pt,english]{amsart}[2004/08/06]
\usepackage{microtype}
\usepackage{fixltx2e}[2006/03/24]
\usepackage[utf8]{inputenc}
\usepackage{svninfo}
\svnInfo$Id$
\usepackage{textcmds}
\usepackage[T1]{fontenc}
\usepackage{lmodern}
\usepackage{textcomp}
\usepackage[headings]{fullpage}
\usepackage{booktabs}
\usepackage{natbib}
\usepackage{graphicx}

\usepackage{amssymb}
\usepackage{mathtools}
\mathtoolsset{mathic, showonlyrefs}
\usepackage{enumitem}
\setenumerate[1]{label=(\arabic*)}
\setenumerate[2]{label=(\alph*), ref=(\theenumi.\alph*)}

\usepackage[english,bookmarksnumbered,bookmarksopen,unicode,pdfusetitle, breaklinks=false,pdfborder={0 0 0},backref=false,colorlinks=false]{hyperref}

\hypersetup{
  pdfauthor={Gábor Braun and Sebastian Pokutta},
  pdfkeywords={random construction, abelian groups with prescribed
  endomorphisms, probabilistic method},
  pdfsubject={MSC2010 Primary: 20K20, 20K15, 20K30, 05D40;
    Secondary: 60B15}} 

\numberwithin{equation}{section}

\newcommand{\mynewtheorem}[2]{%
  \newtheorem{#1}{#2}[section]%
  \expandafter\DeclareRobustCommand\expandafter{\csname#1autorefname\endcsname}{#2}%
  \expandafter\let
    \csname c@#1\expandafter\endcsname
    \expandafter=%
    \csname c@mytheorem\endcsname}
  {NOT TO USE}[section]

\theoremstyle{plain}
\mynewtheorem{theorem}{Theorem}
\mynewtheorem{lemma}{Lemma}
\mynewtheorem{proposition}{Proposition}
\mynewtheorem{corollary}{Corollary}
\mynewtheorem{conjecture}{Conjecture}
\mynewtheorem{note}{Note}

\theoremstyle{definition}
\mynewtheorem{definition}{Definition}
\mynewtheorem{distribution}{Distribution}
\mynewtheorem{example}{Example}
\mynewtheorem{remark}{Remark}
\mynewtheorem{question}{Question}
\mynewtheorem{answer}{Answer}

\newcommand*{\probability}[1]{\operatorname{\mathbb{P}}\left[#1\right]}
\newcommand*{\probabilityGiven}[2]{\operatorname{\mathbb{P}}\left[#1%
  \middle|#2\right]}
\newcommand*{\expectedValue}[1]{\operatorname{\mathbb{E}}\left[#1\right]}
\newcommand*{\expectedValueSub}[2]{\operatorname{\mathbb{E}}_{#1}\left[#2\right]}
\newcommand*{\size}[1]{\left|#1\right|}
\newcommand*{\genBy}[1]{\left\langle#1\right\rangle}
\newcommand*{\genPureBy}[1]{{\genBy{#1}}\sb{*}}
\newcommand*{\genPBy}[2]{{\genBy{#2}}\sb{#1*}}
\newcommand*{\moduleGenBy}[2]{\genBy{#2}\sb{#1}}
\newcommand*{\moduleGenPureBy}[2]{\genBy{#2}\sb{#1,*}}

\newcommand*{\integerPart}[1]{
  \left\lceil
  #1%
  \right\rceil}

\DeclareMathOperator{\End}{End}
\DeclareMathOperator{\Hom}{Hom}
\DeclareMathOperator{\GL}{GL}
\DeclareMathOperator{\length}{length}

\newcommand{\Z}{\mathbb{Z}}

\newcommand{\R}{\mathbb{R}}
\newcommand{\N}{\mathbb{N}}

\begin{document}
\title{Rigid abelian groups and the probabilistic method}

\author{Gábor Braun}
\address{Universität Duisburg–Essen, Campus Essen\\
  Fachbereich Mathematik, AG Göbel–Strüngmann
  Universitätsstraße 2\\
  45117 Essen}
\email{gabor.braun@uni-duisburg-essen.de}

\author{Sebastian Pokutta}
\address{Friedrich-Alexander-University of Erlangen-Nürnberg\\
Am Weichselgarten 9\\
91058 Erlangen\\
Germany}
 \email{sebastian.pokutta@math.uni-erlangen.de}

\date{\svnToday/Draft/Revision: \svnInfoRevision}

\dedicatory{Dedicated to Rüdiger Göbel on the occasion of his 70th birthday.}

\begin{abstract}
  The construction of torsion-free abelian groups with prescribed
  endomorphism rings starting with Corner's seminal work (see \citet{corner}) is a
  well-studied subject in the theory of 
  abelian groups. Usually these construction work by adding elements
  from a (topological) completion in order to get rid of (kill) unwanted
  homomorphisms. The critical part
  is to actually prove that every unwanted homomorphism can be killed
  by adding a suitable element.
  We will demonstrate
  that some of those constructions can be significantly simplified by
  choosing the elements at
  random. As a result, the endomorphism ring will be almost surely
  prescribed, i.e., with probability one. 
\end{abstract}

\maketitle

\section{Introduction}
\label{sec:introduction}
The probabilistic method, pioneered by Erdős (see
\citet{erdos1959,erdos1961}) is one of the most powerful
tools in combinatorics, theoretical computer science, and other
branches of mathematics to show the existence of mathematical objects with prescribed
properties. It is a non-constructive method which infers
the existence of a mathematical object by showing that \emph{the
  probability of its existence} is non-zero. Since its
early days it has lead to a wide range of striking and unexpected
results (cf., e.g.,
\citet{erdos1959random,shelah1988zero,shelah1994random}); for an extensive
overview as well as a very nice introduction the interested reader is
referred to \citet{alon2000probabilistic}. 
We will use the
probabilistic method in order to show the existence of abelian groups
with prescribed endomorphism rings. By doing so, we obtain the
probabilistic counterparts of well-known constructions. While the
statements of the probabilistic counterparts are more general in some
sense, as they assert that almost any choice of, say, elements from
the completion suffice, the proofs simplify.  Another application of the
probabilistic method in abelian group theory,
constructing groups with prescribed Ulm sequences,
was presented in \citet{drogo2010}.

The structure of the paper is as follows. We start with a brief
introduction to the probabilistic method and recall a few concepts
from probability theory in Section~\ref{sec:brief-intr-prob}. We will
then apply the method to construct infinite abelian groups
with prescribed endomorphism rings. For each
construction, we will first recall the
 deterministic construction and provide a sketch of its
proof, then we provide the necessary probabilistic tools and specify
the distributions from which the elements or substructures are drawn,
and finally we present the proof of the probabilistic variant of the
construction. In the first part, in
Section~\ref{sec:groups-via-p}, we consider the classical Corner
construction (see \citet{corner} or \citet{Corner_rigid}). We first show that a uniform, random choice
of countably many \(p\)-adic integers forms an algebraically
independent set with probability one (Lemma~\ref{lem:3}) and later we generalize this
construction to \(2^{\aleph_0}\) elements (Lemma~\ref{lem:1}). We then
provide a probabilistic version of Corner's construction (Theorem~\ref{thm:rigidProb}). In this
case the actual distribution chosen for the random elements does
matter and we provide an example
where using a nearly uniform distribution results in a free group
(Theorem~\ref{th:random-free}). We  then proceed 
with the Zassenhaus construction (see \citet{zassenhaus1967orders}) in
Section~\ref{sec:small-rank-groups} showing that every ring with a
finite-rank free additive group can be realized as the endomorphism
ring of a torsion-free abelian group. While the proof of the
deterministic version (Theorem~\ref{thm:Zassenhaus}) is rather non-trivial and slightly technical,
the proof of the probabilistic version follows more naturally
(Theorem~\ref{thm:ZassenhausProb}) relying on an old result by
Frobenius and Chebotarëv (Lemma~\ref{lem:recSumPrimes}). \medskip

In the following,
let \(\widehat{B}\) denote the \(p\)-adic completion of \(B\).
Let \(J_{p} \coloneqq \widehat{\mathbb{Z}}\)
denote the ring of \(p\)-adic integers.
The \(p\)-adic completion is mainly of interest
when \(B\) is naturally a submodule of \(\widehat{B}\),
which happens exactly when \(B\) is \emph{\(p\)-reduced},
i.e., satisfying \(\bigcap_{n=0}^{\infty} p^{n} B = 0\). Further let
\(X_{p*}\) denote the \emph{\(p\)-purification} of
a submodule \(X\) of a \(p\)-torsion-free module
for some prime \(p\),
i.e., \(X_{p*}\) consists of all \(x/p^{k}\) from the ambient module
with \(x \in X\) and \(k \in \mathbb{N}\).
We omit the ambient module from the notation as it will
be clear from the context.  All other notation is
standard as to be found in \citet{eklof2002almost, jech1978set,
  goebel2006approximations}, and \citet{fuchs1970infinite,
  fuchs1973infinite}.
Recall that an event happens
\emph{almost surely} if the probability of the event is \(1\). For
convenience we define \([n] \coloneqq \{1, \dots, n\}\) for \(n \in \N\). 

\section{The probabilistic method: a brief introduction}
\label{sec:brief-intr-prob}
We will now present a brief introduction to the probabilistic method
and recall the necessary notions and concepts from probability
theory. For a more complete introduction we refer the interested
reader to \citet{alon2000probabilistic}. As mentioned above,
the probabilistic method establishes the existence of structures with
desired properties
by picking the structure randomly
and showing that it has
the desired properties with a positive
probability.
Before we continue with
an example to illustrate the method, we recall a few
notions and concepts from probability theory. 

Recall that probability theory works with a collection of events,
which form a so-called \(\sigma\)-algebra:
it consists of some subsets of a big set closed under countable union
and complements, and therefore also countable intersections.
There is a \emph{probability measure} \(\probability{.}\)
assigning to each event a number in \([0,1]\),
the \emph{probability} of the event.
The probability measure has to satisfy various properties,
from which we mention only
\(\probability{\bigcup_{i < \omega}  A_i} \leq
\sum_{i < \omega} \probability{A_i}\)
for any countable family of events \(A_1, A_2,\dots\).
A collection \(\{A_{i} : i \in J \}\) of events is
\emph{independent}, if 
\(\probability{\bigcap_{i \in I}  A_i} = \prod_{i \in I}
\probability{A_i}\) for any finite \(I \subseteq J\).
The following well-known lemma will be crucial:

\begin{lemma}[Borel-Cantelli Lemma]\label{lem:borelCant} Let \(A_1, A_2, \dotsc \subseteq
\mathcal F\) be a sequence of events. Further let
\(\limsup_{i \rightarrow \infty} A_i\) denote the set of outcomes that
  occur infinitely often. The following hold:
\begin{enumerate}
\item If \(\sum_{i < \omega} \probability{A_i} < \infty\) then
  \(\probability{\limsup_{i \rightarrow \infty} A_i} = 0\).
\item\label{item:1}
  If \(A_1, A_2, \dots\) are independent and \(\sum_{i < \omega}
  \probability{A_i} = \infty\), then
  \(\probability{\limsup_{i \rightarrow \infty} A_i} = 1\). 
  In other words, infinitely many events occur with probability \(1\).
\end{enumerate}
\end{lemma}

A \emph{distribution} of a random variable is
the minimal \(\sigma\)-algebra of events
meaningful for the variable
together with the probability measure on it.
For a discrete random variable \(X\),
i.e., one taking only countably many values,
its
\emph{expected value} is
\(\expectedValue{X} = \sum_{i} X_i \probability{X = X_i} \),
where the \(X_i \in \R\) form the range of \(X\).
Occasionally, we will use expected values of more general variables,
but for intuition, it is
mostly sufficient to think of the expected value in its discrete form. We
will later use Fubini's theorem which allows for iterated computation
of expected values: 

\begin{lemma}[Fubini's Theorem]
  \label{lem:fubini} Let \(f\) be a non-negative function which is measurable (in the respective
  space) and let \(X,Y\) be independent random variables and
\[\expectedValueSub{X,Y}{f(X,Y)} < \omega.\]
Then
\[\expectedValueSub{X,Y}{f(X,Y)} =
\expectedValueSub{X}{\expectedValueSub{Y}{f(X,Y)}}.\]
\end{lemma}
Here expected values are taken in
the total distribution of the variables in the subscript,
and the expected value is a function of the other random variables.

We will now illustrate the probabilistic method by
computing the order of \(\GL(n,q)\),
the group of invertible \(n \times n\) matrices
over the field with \(q\) elements.
This is merely a reformulation of a counting argument
in the framework of probability theory,
just as many early examples.

\begin{proposition}
Let \(n \in \N\) be a natural number
and \(\mathbb F_q\) be the finite field
with \(q\) elements. Then the number of \(n \times n\) invertible
matrices over \(\mathbb F_q\) is 
\[\size{\GL(n,q)} = \prod_{k \in [n]} ( q^{n} - q^{k-1} ).\]
\begin{proof}
Let \(A\)  be a random matrix over \(\mathbb F_q\) chosen with uniform distribution.
Clearly, \(A\) is invertible if and only if its columns \(a_{1}, \dots, a_{n}\) 
are linearly independent. Observe that the probability of \(A\) being
invertible can be rephrased by breaking it up into
probabilities of linear independence of smaller subsets:
\begin{multline}
  \label{eq:33}
  \probability{A \text{ invertible}} =
  \prod_{k \in [n]}
  \probabilityGiven{a_{1}, \dots, a_{k} \text{ independent}}%
  {a_{1}, \dots, a_{k-1} \text{ independent}}
  \\
  =
  \prod_{k \in [n]}
  \left(
    1 - \probabilityGiven{a_{k} \in \genBy{a_{1}, \dots, a_{k-1}}}{a_{1},
    \dots, a_{k-1} \text{ independent}}
  \right)
\end{multline}
Provided that \(a_{1}, \dots, a_{k-1}\) are linearly independent,
they span a (\(k-1\))-dimensional subspace, so the probability that
\(a_{k}\) is in this subspace is
\begin{equation}
  \label{eq:32}
  \probabilityGiven{a_{k} \in \genBy{a_{1}, \dots, a_{k-1}}}{a_{1},
    \dots, a_{k-1} \text{ independent}} = \frac{q^{k-1}}{q^{n}},
\end{equation}
as the columns are independent random variables. We therefore obtain
\begin{equation}
  \label{eq:34}
  \probability{A \text{ invertible}} =
  \prod_{k \in [n]}
  \left(
    1 - \frac{q^{k-1}}{q^{n}}
  \right).
\end{equation}
On the other hand we have \(\probability{A \text{ invertible}} =
\frac{\ell}{q^{n^2}}\), where \(\ell\) is the number of invertible
matrices and \(q^{n^2}\) is the total number of \(n \times n\)
matrices over \(\mathbb F_q\). We therefore obtain 
\[\ell =   q^{n^2} \cdot \prod_{k \in [n]}
  \left( 1 - \frac{q^{k-1}}{q^{n}} \right)  = \prod_{k \in [n]}
   q^{n} \left( 1 - \frac{q^{k-1}}{q^{n}} \right) = \prod_{k \in [n]} ( q^{n} - q^{k-1} ).\]
\end{proof}
\end{proposition}

In the following we operate under the same paradigm. However, it is
not the abelian groups \emph{per se} that are drawn from random
distributions. We will use the concept in a slightly different
fashion: we will pick crucial elements of the constructions, such as
elements from the completion, at random.
Obviously, we have to specify \emph{how} we actually pick these
elements, i.e., we have to provide the distribution.
The distributions that we will use are very natural and since
we are concerned about existence only, we can basically pick any
(well-defined) distribution that suits our needs. 

Another fact that is worthwhile to be mentioned is the structure of our
results. We do not just provide mere \emph{existence statements}, but we will
show that the endomorphism properties hold \emph{almost surely}, i.e.,
every random choice is satisfactory with probability \(1\). 
Actually, this is expected in view of Kolmogorov's zero-one law.

\section{Groups via \(p\)-adic numbers}
\label{sec:groups-via-p}
Our starting point is the following well-known construction of
Corner (see~\citet{corner} or~\citet[Theorem~110.1]{fuchs1973infinite}).
For simplicity, we restrict to \(p\)-reduced rings \(R\).

\begin{theorem}
  \label{thm:rigid}
  For every countable \(p\)-reduced torsion-free ring \(R\),
  there is a torsion-free left abelian group of countably infinite rank
  with endomorphism ring \(R\).
\begin{proof}[Sketch of proof] 
Let \(\xi_{n} \in J_{p}\) with \(n < \omega\) be quadratically
independent
\(p\)-adic integers and further let \(B\) be a free \(R\)-module of
countably infinite rank.
We define
\begin{equation}
  \label{eq:1}
  G \coloneqq
  \genPBy{p}{%
    B, R b \xi_{b} : b \in B \setminus \{0\}
  } \subseteq \widehat{B}.
\end{equation}
Then \(\End G = R\).
For details, see \citet{corner},
or for a slightly different construction
\citet[Theorem~110.1]{fuchs1973infinite}, or the
proof of Theorem~\ref{thm:rigidProb} below.
\end{proof}
\end{theorem}

Note that the construction in Theorem~\ref{thm:rigid} carries over
to uncountable modules up to size \(2^{\aleph_{0}}\). In order to establish the probabilistic version,
we will choose continuum many random \(p\)-adic integers,
which will be almost surely algebraically independent. First we present the easier, countable case:
countably many, randomly and independently chosen \(p\)-adic
integers are almost surely algebraically independent. 
\begin{lemma}
  \label{lem:3}
Let \(\mathcal M = \{\xi_n \mid n < \omega\} \subseteq J_p\) be a set of 
  countably many, randomly and independently chosen \(p\)-adic
  integers such that \(\probability{\xi_{n} = \lambda} = 0\)
  for every \(n < \omega\) and \(\lambda \in J_p\). Then \(\mathcal M\) is almost surely algebraically independent.
\begin{proof}
We show that every finite subset \(S \subseteq \mathcal M\) is
almost surely algebraically
independent. The proof is by induction on the cardinality \(n\) of
\(S\). For \(n = 0\) the statement holds trivially as \(S =
\emptyset\). Therefore let \(n \geq 1\) and let \(S = \{\xi_{1}, \dots,
\xi_{n}\}\) be a finite subset of \(\mathcal M\). Note that 
there are only countably many non-zero polynomials \(f\)
with integer coefficients in \(n\) variables. Thus it suffices to show
that \(f(\xi_{1}, \dots, \xi_{n}) \neq 0\)
almost surely for every such \(f\). By assumption
\(\xi_{n}\) is independent of \(\xi_{1}, \dots, \xi_{n-1}\). We can
therefore apply Lemma~\ref{lem:fubini} to compute the probability
\(\probability{f(\xi_{1}, \dots, \xi_{n}) \neq 0}\) by iterating
expected values:
\begin{equation}
  \label{eq:30}
  \probability{f(\xi_{1}, \dots, \xi_{n}) \neq 0} =
  \expectedValueSub{\xi_1, \dots, \xi_{n-1}}{%
    \probability{f(\lambda_{1}, \dots, \lambda_{n-1}, \xi_{n}) \neq 0 \mid
      \xi_{i} = \lambda_{i}, i \in [n-1]}%
  }. 
\end{equation}
By induction, we conclude that \(f(\xi_{1}, \dots, \xi_{n-1}, x_{n})\) is almost surely
a non-zero polynomial in \(x_n\). Therefore it has only finitely many
roots and together with the assumption \(\probability{\xi_{n} = \lambda} = 0\)
  for every \(n < \omega\) and \(\lambda \in J_p\), we infer that
  \(\xi_{n}\) is none of these roots almost surely. It follows that 
\[\expectedValueSub{\xi_1, \dots, \xi_{n-1}}{%
    \probability{f(\lambda_{1}, \dots, \lambda_{n-1}, \xi_{n}) \neq 0 \mid
      \xi_{i} = \lambda_{i}, i \in [n-1]}%
  } = 1\]
which completes the proof. 
\end{proof}
\end{lemma}

Note that
quadratic independence instead of algebraic independence
can be easily shown without the use of Fubini's Theorem.

A slightly
more involved construction allows us to choose even
continuum many random \(p\)-adic numbers, 
which are almost surely algebraically independent. We hasten to emphasize
a peculiarity of the statement:
it states that almost always none of
\emph{uncountably} many events occur.
Usually probability
theory cannot provide an answer in such cases as it only
asserts that the union of \emph{countably} many probability-\(0\) events
has again probability \(0\). However here we can use that
\(J_p\) is compact, hence we can \emph{approximate} the events
via the topology.  To ensure this, we construct the numbers as
infinite branches of a tree and
we show that it suffices
to confine ourselves to sufficiently long \emph{finite} initial
segments. By doing so we reduce the uncountable case to a countable
one. The construction is similar to the one in
\citet{Corner_rigid}. Let \(\length(s)\) denote the length of a
sequence \(s\). 

\begin{lemma}
  \label{lem:1}
  Let \(p\) be an integer.
  We construct \(2^{\aleph_0}\) random \(p\)-adic numbers as follows.
  We choose randomly and independently non-negative integers
  \(a_{s} \in \{0, 1, \dots, p^{2^{n+1} - 2^{n}} - 1\}\)
  with uniform distribution
  for every finite 0-1 sequence \(s\)
  where \(n = \length(s)\).
  In particular,
  \(a_{\langle \rangle} \in \{0, 1, \dots, p-1\}\)
  for the empty sequence \(\langle \rangle\).
  For every 0-1 infinite sequence \(f\),
  we define the \(p\)-adic number
  \begin{equation}
    \label{eq:24}
    \xi_{f} \coloneqq \sum_{n=0}^{\infty} p^{2^{n} - 1} a_{f \restriction n},
  \end{equation}
  where \(f \restriction n\) is the initial segment of \(f\)
  consisting of \(n\) elements.
  Then the \(\xi_{f}\) are almost surely algebraically independent.
\begin{proof}
To handle the \(\xi_{f}\) more easily we define
\begin{equation}
  \label{eq:29}
  b_{s} \coloneqq \sum_{j=0}^{n} p^{2^{j} - 1} a_{s \restriction j}
\end{equation}
for every finite 0-1 sequence \(s\)
where \(n \coloneqq \length(s)\).
Then we have
\begin{equation}
  \label{eq:28}
  \xi_{f} \equiv b_{f \restriction n} \pmod{p^{2^{n+1} - 1}}.
\end{equation}

First note that every \(b_{s}\) is uniformly distributed
on the set of integers \(\{0, 1, \dots, p^{2^{n+1} - 1}\}\),
i.e., on the mod \(p^{2^{n+1} - 1}\) classes of \(J_{p}\)
with \(n \coloneqq \length(s)\).
This implies, in particular,
\begin{equation}
  \label{eq:18}
  \probabilityGiven{b_{s} \equiv c \pmod{p^{2^{n+1} - 1}}}%
  {b_{s \restriction j} \equiv d \pmod{p^{2^{j+1} - 1}}}
  =
  \begin{cases}
    \frac{1}{p^{2^{n+1} - 2^{j+1}}}, &
    c \equiv d \pmod{p^{2^{j+1} - 1}},
    \\
    0, & \text{otherwise.}
  \end{cases}
\end{equation}

For every positive integers \(k\) and \(n\),
every non-zero polynomial \(g\)
with integer coefficients in \(k\) variables,
and every \emph{pairwise distinct} finite 0-1 sequences
\(s_{1}, \dots, s_{k}\) of length \(n\),
we show that there is almost never an extension
\(f_{i}\) of the \(s_{i}\) with \(g(\xi_{f_{1}}, \dots, \xi_{f_{k}})=0\).
This will prove the lemma,
as these are altogether countably many events,
whose union is therefore the probability-\(0\) event that
the \(\xi_{f}\) are dependent.

We use induction on \(k\).
The statement for \(k=0\) is obvious.
For \(k > 0\),
we prove the claim by showing that
the probability of the event is at most \(\varepsilon\)
for all positive \(\varepsilon > 0\). Let \(\mu\) denote the Haar probability measure of
the compact additive group \(J_{p}^{k-1}\).
We say that a subset \(A \subseteq J_{p}^{k-1}\) is \emph{admissible}
if the event that
\(g\) has a solution \(\xi_{f_{1}}, \dots, \xi_{f_{k}}\)
for some infinite 0-1 sequences \(f_{i}\) extending the \(s_{i}\)
with \((\xi_{f_{1}}, \dots, \xi_{f_{k-1}}) \in A\)
has probability at most \(\varepsilon \mu(A)\).
We will prove the claim by partitioning \(J_{p}^{k-1}\)
into countably many admissible subsets. 

For this, write \(g\) in the form:
\begin{equation}
  \label{eq:25}
  g(x_{1}, \dots, x_{k}) = g_{m}(x_{1}, \dots, x_{k-1}) x_k^{m} + \dots +
  g_{0}(x_{1}, \dots, x_{k-1}),
\end{equation}
where \(g_{m} \neq 0\). We choose one of the partitions to be the solution set of \(g_{m}\),
which is admissible (actually has probability \(0\))
by the induction hypothesis on \(k\).
The other partitions will be basic open sets, i.e.,
mod \(p^{N}\)-classes.
As there are only countably many mod \(p^{N}\)-classes
and every family of such classes contains a pairwise disjoint
subfamily with the same union,
it is enough to prove that every
\((\eta_{1}, \dots, \eta_{k-1}) \in J_{p}^{k-1}\)
with \(g_{m}(\eta_{1}, \dots, \eta_{k-1}) \neq 0\)
is contained in an admissible mod \(p^{N}\)-class for some \(N\).
Actually, we show that the mod \(p^{N}\)-class \(A\) of 
\((\eta_{1}, \dots, \eta_{k-1}) \in J_{p}^{k-1}\)
is admissible for \(N\) large enough,
because even the event that
there are extensions \(f_{i}\) of the \(s_{i}\)
with \(\xi_{f_{i}} \equiv \eta_{i} \pmod{p^{N}}\)
and
\(\xi_{f_{1}}, \dots, \xi_{f_{k}}\) is
a solution of \(g\) mod \(p^{N}\),
i.e., \(g(\eta_{1}, \dots, \eta_{k-1}, \xi_{f_{k}}) \equiv 0 \pmod{p^{N}}\)
has probability at most \(\varepsilon \mu(A)\).

Let \(f \succ s\) denote that
the sequence \(f\) is an extension of \(s\).
We consider the probability modulo the values of the \(b_{s_{i}}\),
as this makes the conditions on the \(\xi_{f_{i}}\) independent:
\begin{multline}
  \label{eq:27}
  \probabilityGiven{%
    \exists f_{i} \succ s_{i}: \xi_{f_{i}} \equiv \eta_{i} \pmod{p^{N}},
    g(\eta_{1}, \dots, \eta_{k-1}, \xi_{f_{k}}) \equiv 0 \pmod{p^{N}}}%
  {b_{s_{1}}, \dots, b_{s_{k}}}
  \\ =
  \prod_{i \in [k-1]}
  \probabilityGiven{%
    \exists f_{i} \succ s_{i}:
    \xi_{f_{i}} \equiv \eta_{i} \pmod{p^{N}}}%
  {b_{s_{1}}, \dots, b_{s_{k}}}
  \\ \cdot
  \probabilityGiven{%
    \exists f_{k} \succ s_{k}:
    g(\eta_{1}, \dots, \eta_{k-1}, \xi_{f_{k}}) \equiv 0
    \pmod{p^{N}}}%
  {b_{s_{1}}, \dots, b_{s_{k}}}.
\end{multline}
Let us fix a positive integer
\(r \coloneqq \integerPart{\log_{2} (N  + 1) - 1} = O(\log N)\)
so that \(\xi_{f} \equiv b_{f \restriction r} \pmod{p^{N}}\)
for every infinite 0-1 sequence \(f\).
For every \(i \in [k]\),
there are \(2^{r-n}\) extensions of \(s_{i}\) into a
0-1 sequence of length \(r\) where \(n\) is the length of the \(s_i\).
Therefore the probability that there exists an extension which is
equivalent to \(\eta_{i}\) mod \(p^{N}\)
is at most
\begin{equation}
  \label{eq:20}
  \probabilityGiven{%
    \exists f_{i} \succ s_{i}:
    \xi_{f_{i}} \equiv \eta_{i} \pmod{p^{N}}}%
  {b_{s_{1}}, \dots, b_{s_{k}}}
  \leq
  \frac{2^{r-n}}{p^{N - 2^{n+1}}}.
\end{equation}
For \(i=k\), by a similar argument,
\begin{equation}
  \label{eq:21}
  \probabilityGiven{%
    \exists f_{k} \succ s_{k}:
    g(\eta_{1}, \dots, \eta_{k-1}, \xi_{f_{k}}) \equiv 0
    \pmod{p^{N}}}%
  {b_{s_{1}}, \dots, b_{s_{k}}}
  \leq
  \frac{2^{r-n} R(N)}{p^{N - 2^{n+1}}}, 
\end{equation}
where \(R(N)\) is the number of roots of
\(g(\eta_{1}, \dots, \eta_{k-1}, x)\) in \(x\) mod \(p^{N}\).

To estimate \(R(N)\),
let us consider the factorization over \(J_{p}\)
\begin{equation}
  \label{eq:26}
  g(\eta_{1}, \dots, \eta_{k-1}, x) =
  h(x) \prod_{i \in [l]} (x - \lambda_{i})
\end{equation}
for some \(p\)-adic integers \(\lambda_{i}\).
The polynomial \(h\) has no roots among the \(p\)-adic integers.
So there is a highest \(p\)-power \(p^{M}\)
which can divide \(h(x)\) for any \(p\)-adic number \(x\).

Let us estimate the number of roots of \eqref{eq:26}
modulo \(p^{N}\).
For every root \(x\), the product is divisible by \(p^{N}\).
As \(h(x)\) is divisible by at most \(p^{M}\),
there must be an \(i\) for which \(x - \lambda_{i}\)
is divisible by \(p^{\integerPart{(N - M)/l}}\).
So every root is contained in
the mod \(p^{\integerPart{(N - M)/l}}\)-class of
some \(\lambda_{i}\), and hence
\begin{equation}
  \label{eq:22}
  R(N) \leq 
l
p^{N - \integerPart{(N - M)/l}}.
\end{equation}

By combining \eqref{eq:27}, \eqref{eq:20}, \eqref{eq:21} and
\eqref{eq:22},
we finally obtain
\begin{multline}
  \label{eq:23}
  \probabilityGiven{%
    \exists f_{i} \succ s_{i}: \xi_{f_{i}} \equiv \eta_{i} \pmod{p^{N}},
    g(\eta_{1}, \dots, \eta_{k-1}, \xi_{f_{k}}) \equiv 0 \pmod{p^{N}}}%
  {b_{s_{1}}, \dots, b_{s_{k}}}
  \\ \leq
  {\left(
    \frac{2^{r-n}}{p^{N - 2^{n+1}}}
  \right)}^{k-1}
  \cdot
  \frac{2^{r-n} l p^{N - \integerPart{(N - M)/l}}}{p^{N - 2^{n+1}}}
  \\ =
  \frac{
    l
    {\left(
        2^{r-n} p^{2^{n+1}}
      \right)}^{k}}%
   {p^{\integerPart{(N - M)/l}}}
  \cdot
  \underbrace{\frac{1}{p^{N (k-1)}}}_{\mu(A)}
  =
  O \left( \frac{N^{k}}{p^{N/l}} \right)
  \cdot
  \mu(A).
\end{multline}
Hence \(A\) is indeed admissible for large \(N\).
\end{proof}
\end{lemma}

For a countable module \(B\),
we will randomly and independently choose elements
\(a_n = \sum_{b \in I_{n}} b \xi_{n,b} \in J_{p} B\)
for all \(n < \omega\).
To this end, we select the support \(I_{n}\) and
the coefficients \(\xi_{n,b}\)
according to the following distribution.

\begin{distribution}
\label{dist:corner}
For a countable set \(B\) and for \(n < \omega\),
let \(I_{n}\) be independent, identical distributed
random variables taking values in the 
non-empty finite subsets of \(B\).
Every non-empty finite subset should be contained in \(I_{n}\)
(for a fixed \(n\))
with positive probability.

Furthermore,
for all \(n\) and \(b \in I_{n}\) and \(\alpha < 2^{\aleph_{0}}\)
let the \(\xi_{n,b}^{\alpha}\)
be random \(p\)-adic numbers
chosen as in Lemma~\ref{lem:1}.
Note that the \(\xi_{n,b}^{\alpha}\) are almost surely
algebraically independent 
for all \(n < \omega, b \in I_n, \alpha < 2^{\aleph_{0}}\).
\end{distribution}

We can prove the following probabilistic
variant of Theorem~\ref{thm:rigid}. 

\begin{theorem}
  \label{thm:rigidProb}
  Let \(R\) be a countable \(p\)-reduced, torsion-free ring.
  Let \(B\) be an at most countably generated,
  non-zero, free \(R\)-module.
  Furthermore, let \(I_{n}\) be random finite subsets of \(B\)
  and \(\xi_{n,b}^{\alpha}\) for \(b \in I_{n}\) and
  \(\alpha < 2^{\aleph_{0}}\)
  be random \(p\)-adic numbers
  with Distribution~\ref{dist:corner}, and define
  \begin{equation}
    \label{eq:6}
    a_{n}^{\alpha} \coloneqq
    \sum_{b \in I_{n}} b \xi_{n,b}^{\alpha} \in J_p B.
  \end{equation}
  Then the groups
  \[  G^{A} \coloneqq
  \genPBy{p}{
    B, R a_{n}^{\alpha} : n < \omega, \alpha \in A
    } \subseteq \widehat{B}.
  \]
  for \(\emptyset \neq A \subseteq 2^{\aleph_{0}}\)
  have endomorphism ring \(\End G^{A} = R\)
  and form a fully rigid system, i.e.,
  \begin{equation}
    \label{eq:31}
    \Hom(G^{A}, G^{D}) =
    \begin{cases}
      R, & A \subseteq D \\
      0, & A \nsubseteq D
    \end{cases}
  \end{equation}
  almost surely.
\begin{proof}
By \autoref{lem:1} the family
\(\{\xi_{n,b}^{\alpha} \mid n < \omega, b \in I_n,
\alpha < 2^{\aleph_{0}}\}\)
is almost surely algebraically independent. 
Moreover, every finite \(F \subseteq B\)
is almost surely contained in some (actually infinitely many)
\(I_{n}\) with \(n < \omega\).
We will show that these two properties guarantee that
\(\Hom(G^{A}, G^{D})\) is \(R\) or \(0\)
almost surely, as claimed, i.e.,
all homomorphisms are multiplications by ring elements.

Let \(\varphi\) be a homomorphism from \(G^{A}\) to \(G^{D}\) and let
\(\alpha \in A \subseteq 2^{\aleph_{0}}\) be arbitrary but fixed for the moment.
Obviously, \(b \varphi, a_{n}^{\alpha} \varphi \in G^{D}\) for \(b \in B\)
so there are \(d_{b}, c_{n}, \in \mathbb{Z}[1/p] B\) and
\(t_{m,b}, r_{m,n} \in R[1/p]\)
together with \(\beta_{m,b}, \delta_{m,n} \in D\)
such that
\begin{align}
  \label{eq:7}
  b \varphi = d_{b} + \sum_{m} t_{m,b} a_{m}^{\beta_{m,b}}
  &= d_{b} + \sum_{m,f \colon f \in I_{m}} t_{m,b} f \xi_{m,f}^{\beta_{m,b}}, \\
  \label{eq:9}
  a_{n}^{\alpha} \varphi &= c_{n} +
  \sum_{m,f \colon f \in I_{m}} r_{m,n} f \xi_{m,f}^{\delta_{m,n}}.
\end{align}
On the other hand, by continuity, we also obtain from \eqref{eq:6} and \eqref{eq:7}
\begin{equation}
  \label{eq:8}
  a_{n}^{\alpha} \varphi =
  \sum_{b \in I_{n}}
  d_{b} \xi_{n,b}^{\alpha} +
  \sum_{\substack{m,f \colon f \in I_{m}, \\ b \in I_{n}}}
  t_{m,b} f \xi_{m,f}^{\beta_{m,b}} \xi_{n,b}^{\alpha}.
\end{equation}
By combining the two expressions for \(a_{n}^{\alpha} \varphi\)
we therefore obtain
\begin{equation}
  \label{eq:10}
  \sum_{b \in I_{n}}
  d_{b} \xi_{n,b}^{\alpha} +
  \sum_{\substack{m,f \colon f \in I_{m}, \\ b \in I_{n}}}
  t_{m,b} f \xi_{m,f}^{\beta_{m,b}} \xi_{n,b}^{\alpha}
  =
  c_{n} +
  \sum_{m,f \colon f \in I_{m}} r_{m,n} f \xi_{m,f}^{\delta_{m,n}}.
\end{equation}
Using the algebraic independence of the \(\xi_{n,b}^{\gamma}\),
we compare coefficients and obtain among others
\begin{align}
  \label{eq:11}
  t_{m,b} f &= 0, & (f &\in I_{m}) \\
  \label{eq:14}
  d_{b} &= r_{n,n} b & (b &\in I_{n}) \\
  \label{eq:17}
  d_{b} &= 0 & (\alpha &\notin D).
\end{align}
We have used that for every \(b \in B\)
there is an \(n\) with \(b \in I_{n}\).
For example, to obtain the first equation, we choose \(n \neq m\)
with \(b \in I_{n}\) and compare the coefficients of
\(\xi_{m,f}^{\beta_{m,b}} \xi_{n,b}^{\alpha}\). We conclude that if 
\(\alpha \notin D\) then \(b \varphi = 0\) for all \(b \in B\)
and hence \(\varphi = 0\).
This is enough for the case \(A \nsubseteq D\).
If \(\alpha \in D\) then \(b \varphi = d_{b} = r_{n,n} b\)
for all \(n\) and \(b \in I_{n}\).
We now show that essentially all the \(r_{n,n}\) are equal,
i.e., \(b \varphi = r b\) for some \(r \in R[1/p]\). As \(B\) is free,
there is an element \(b'\) with zero annihilator,
e.g., a basis element.
As a consequence, all the \(r_{n,n}\) are equal
for which \(b' \in I_{n}\).
Let \(r\) be the common value of these \(r_{n,n}\), choose \(b \in B\)
arbitrary and pick \(n\)
with \(b', b \in I_{n}\),
which exists by hypothesis.
So \(r=r_{n,n}\), and using \eqref{eq:14} we obtain
\(b \varphi = r_{n,n} b = r b\) as claimed.
We therefore conclude that
the homomorphism \(\varphi\) is multiplication by
an \(r \in R[1/p]\).

As \(B\) is free, \(R[1/p] \cap \End B = R\), and
it follows that \(r \in R\)
and thus \(\varphi\) is a multiplication with the ring element \(r\).
This finishes the case \(A \subseteq D\) and hence the proof.
\end{proof}
\end{theorem}

We also obtain a probabilistic version of Corner's construction
of finite-rank groups as a corollary
(see \citet[Theorem B]{corner} or \citet[Corollary
12.1.3]{goebel2006approximations}).

\begin{corollary}
\label{cor:cornerB}
Let \(A\) be a \(p\)-reduced, \(p\)-torsion-free ring of finite rank
\(n\). Then 
\[G \coloneqq \genPBy{p}{A,wA}\]
is of rank \(2n\) and \(\End(G) \cong A\) almost surely. 
\end{corollary}

In the usual way Theorem~\ref{thm:rigidProb} and
Corollary~\ref{cor:cornerB} can be generalized to \(\mathbb
S\)-reduced, \(\mathbb S\)-torsion free algebras \(A\) of finite rank
 over some \(\mathbb S\)-ring \(R\) whose completion \(\widehat{R}\)
has sufficiently high transcendence degree; we confined ourselves to
the simplified case purely for expository reasons and the
generalization is left to the interested reader.  

Note that the elements \(a_n^{\alpha} \in \widehat{B}\) that we chose at
random in Theorem~\ref{thm:rigidProb} were contained in the submodule \(J_p B\).
It would be natural to expect that a nearly uniform choice of random
elements from the completion \(\widehat{B}\) should already suffice.
However, it fails:
the constructed group is actually almost surely free. This shows in a
nice way that the actual distribution does matter which is somewhat
counterintuitive. It seems that especially the implicit assumption of
finite support in Distribution~\ref{dist:corner} is advantageous.

\begin{theorem}
  \label{th:random-free}
  Let \(B\) be a free abelian group of countably infinite rank and
  let
  \begin{equation}
    \label{eq:5}
    G \coloneqq
    \genPBy{p}{
        B, a_{n}
      } \subseteq \widehat{B}
  \end{equation}
  where the \(a_{n}\) with \(n < \omega\) are independent, random
  elements chosen with a nearly uniform distribution
  from \emph{the completion \(\widehat{B}\)},
  i.e., for some \(\alpha > 1\), 
  all \(n < \omega\), and \(x \in B / p^{n} B\)
  we have
  \(\probability{a_{m} + p^{n} B = x} \leq p^{- n^{\alpha}}\).
  Then \(G\) is almost surely free.
\begin{proof}
First we claim that
the random elements \(a_{m}\) are almost never contained in
the \(J_{p}\)-module generated by
any fixed \(b_{1}, \dots, b_{k} \in \widehat{B}\), i.e.,
\begin{equation}
  \label{eq:2}
  \probability{a_{m} \in \moduleGenBy{J_{p}}{b_{1}, \dots, b_{k}}} = 0.
\end{equation}
The event is the intersection of the descending sequence of events
that \(a_{m}\) is contained in the subgroup generated by the \(b_{i}\)
in the factor group \(\widehat{B} / p^{n} \widehat{B}\).
We estimate the probability of the latter events:
\begin{equation}
  \label{eq:12}
  \probability{a_{m} \in \genBy{b_{1}, \dots, b_{k}}
    + p^{n} B} \leq
  \size{\genBy{b_{1}, \dots, b_{k}} \bmod p^{n}} \cdot
  \frac{1}{p^{n^{\alpha}}} \leq
  \frac{p^{nk}}{p^{n^{\alpha}}}.
\end{equation}
This tends to zero as \(n\) goes to infinity, proving the claim.

Next we show that the family of all the \(e_{n}\) and \(a_{m}\)
is almost surely linearly independent over \(J_{p}\).
If the family is linearly dependent then
\begin{equation}
  \label{eq:13}
  \sum_{i=0}^{k} \eta_{i} e_{i} + \sum_{j=0}^{l} \mu_{j} a_{j} = 0
\end{equation}
for some \(p\)-adic integers \(\eta_{i}\) and \(\mu_{j}\),
where not all of those are zero.
The \(e_{i}\) form a basis of \(B\),
so they remain linearly independent over \(J_{p}\),
hence there must be a non-zero \(\mu_{j}\).
Since \(J_{p}\) is a discrete valuation domain,
there is a \(\mu_{m}\) dividing all the \(\mu_{j}\).
It follows that \(\mu_{m}\) divides \(\sum_{i=0}^{k} \eta_{i} e_{i}\).
Because \(J_{p} B = \bigoplus_{n=0}^{\infty} J_{p} e_{i}\)
is pure in \(\widehat{B}\),
the number \(\mu_{m}\) must divide all of the \(\eta_{i}\).
All in all, we obtain
\begin{equation}
  \label{eq:15}
  a_{m} = - \sum_{i=0}^{k} \mu_{m}^{-1} \eta_{i} e_{i}
  - \sum_{j=0}^{l} \mu_{m}^{-1} \mu_{j} a_{j}
\end{equation}
and therefore
\begin{equation}
  \label{eq:16}
  a_{m} \in
  \moduleGenBy{J_{p}}{%
    e_{0}, \dots, e_{k}, a_{0}, \dots, a_{m-1}, a_{m+1}, \dots, a_{l}
  }.
\end{equation}
Since the \(a_{j}\) are independent random variables,
the event \eqref{eq:16}
has probability zero by \eqref{eq:2}
for fixed \(m\), \(k\), and \(l\).
Varying \(m\), \(k\), and \(l\),
there are only countably many such events,
so almost surely none of them occurs,
and hence the family of all the \(e_{n}\) and \(a_{m}\)
are almost surely linearly independent over \(J_{p}\).

Finally, we show that the linear independence ensures that \(G\) is free.
Recall that \(G\) is countable,
and hence we can apply Pontryagin's criterion
(see \citet[Theorem~19.1]{fuchs1970infinite}
or \citet[Theorem~2.3]{eklof2002almost}):
a countable torsion-free abelian group is free
if and only if every finite-rank subgroup is free.
Therefore it suffices to show that the purifications
\(\genPureBy{e_{0}, \dots, e_{k}, a_{0}, \dots, a_{m}}\)
are actually free groups. Recall that \(\widehat{B}\) as a \(J_{p}\)-module has the property
that every pure finite-rank submodule is a free module.
Therefore
\(\moduleGenPureBy{J_{p}}{e_{0}, \dots, e_{k}, a_{0}, \dots, a_{m}}\)
is free, and in particular for some \(k > 0\), we have
\(p^{k} \moduleGenPureBy{J_{p}}{e_{0}, \dots, e_{k}, a_{0}, \dots, a_{m}}
\subseteq
\moduleGenBy{J_{p}}{e_{0}, \dots, e_{k}, a_{0}, \dots, a_{m}}\).
It follows that every element of
\(p^{k} \genPureBy{e_{0}, \dots, e_{k}, a_{0}, \dots, a_{m}}\)
is a linear combination of the
\(e_{0}, \dots, e_{k}, a_{0}, \dots, a_{m}\)
with coefficients in \(J_{p}\). On the other hand, the coefficients are also in \(\mathbb{Z}[1/p]\),
since
\(p^{k} \genPureBy{e_{0}, \dots, e_{k}, a_{0}, \dots, a_{m}}\)
is a subgroup of \(G\).
All in all, using linear independence, the coefficients are in
\(J_{p} \cap \mathbb{Z}[1/p] = \mathbb{Z}\),
so
\(p^{k} \genPureBy{e_{0}, \dots, e_{k}, a_{0}, \dots, a_{m}}\)
is contained in
\(\genBy{e_{0}, \dots, e_{k}, a_{0}, \dots, a_{m}}\).
Therefore
\(\genPureBy{e_{0}, \dots, e_{k}, a_{0}, \dots, a_{m}}\)
must be free.
\end{proof}
\end{theorem}

\section{Small-rank groups}
\label{sec:small-rank-groups}
In this section we provide a probabilistic counterpart for
Zassenhaus's construction (see \citet{zassenhaus1967orders} or
\citet[Theorem 12.1.6]{goebel2006approximations}). 

\begin{theorem}
  \label{thm:Zassenhaus}
  Let \(A\) be a ring with a finite-rank free additive group.
  Then there is a torsion-free abelian group \(M\) of the same rank
  with endomorphism ring \(A\).
\end{theorem}

\begin{proof}[Sketch of proof]
For every pair of non-zero elements \(a_{i}\), \(e_{i}\) of \(A\),
choose an integer \(c_{i}\) and a prime \(p_{i}\) such that
\(c_{i} - a_{i}\) is invertible in \(\mathbb{Q} A\),
and \(p_{i}\) divides the order of \({(c_{i} - a_{i})}^{-1} e_{i}\) in
\(\mathbb{Q} A / A\).
Make the choices such that the primes \(p_{i}\) are pairwise distinct.
We choose positive integers \(r_{i}, d_{i}\) such that
\(p_{i}^{r_{i}} d_{i} {(c_{i} - a_{i})}^{-1} \in A\)
and \(d_{i}\) is relative prime to \(p_{i}\).
Now the abelian group
\begin{equation}
  \label{eq:3}
  M \coloneqq \genBy{ A, p_{i}^{- r_{i}} (c_{i} - a_{i}) A : i < \omega }
  \subseteq \mathbb{Q} A
\end{equation}
has endomorphism ring \(A\) acting on it
by multiplication on the right.

To see this,
first we note that
for every \(m \in \mathbb{Q}A\)
with \(p_{i}^{- r_{i}} (c_{i} - a_{i}) m \in M\),
the order of
\(m\)
in \(\mathbb{Q} A / A\) is not divisible by \(p_{i}\).
Indeed, there are \(a, b_{j} \in A\) such that
\begin{equation*}
  p_{i}^{- r_{i}} (c_{i} - a_{i}) m =
  a + \sum_{j} p_{j}^{- r_{j}} (c_{j} - a_{j}) b_{j}.
\end{equation*}
Multiplying by
\(\tilde{a} \coloneqq p_{i}^{r_{i}} d_{i} {(c_{i} - a_{i})}^{-1} \in A\)
on the left
\begin{equation*}
  d_{i} m =
  d_{i} b_{i} + \tilde{a} a +
  \sum_{j \neq i} p_{j}^{- r_{j}} \tilde{a} (c_{j} - a_{j}) b_{j}.
\end{equation*}
The order of the right-hand side is clearly
not divisible by \(p_{i}\).
As \(d_{i}\) is not divisible by \(p_{i}\),
it also follows that \(p_{i}\) does not divide
the order of
\(m\).

We will now prove that
the only \(\phi \in \End M\)
mapping \(1\) to \(0\)
is the zero map.
Suppose for contradiction that
there is a non-zero \(\phi\) with \(1 \phi = 0\). Thus there exists
\(a \in A\) with \(a \phi \neq 0\).
By multiplying \(a\) with a large positive integer if necessary,
we can assume \(a \phi \in A\).
In fact there is an \(i\) with
\(a_{i} = a\) and \(-e_{i} = a \phi\). 

Since
\begin{equation*}
  p_{i}^{-r_{i}} e_{i} = p_{i}^{- r_{i}} (c_{i} - a_{i}) \phi
  \in M,
\end{equation*}
the order of
\({(c_{i} - a_{i})}^{-1} e_{i}\)
in \(\mathbb{Q} A / A\)
is not divisible by \(p_{i}\)
contradicting one of our assumptions.

Next we establish that \(\mathbb{Q}A \cap \End M = A\),
where every \(a \in \mathbb{Q}A\) is identified with
multiplication by \(a\) on the right.
So let \(m \in \mathbb{Q}A \cap \End M\).
Then \(m = 1 m \in M\),
so the order of \(m\) in \(\mathbb{Q} A / A\)
can have only the \(p_{i}\) as prime divisors.
On the other hand,
since
\(p_{i}^{- r_{i}} (c_{i} - a_{i}) m \in M\),
the order of \(m\) in \(\mathbb{Q} A / A\)
is not divisible by \(p_{i}\).
Hence the order of \(m\) must be \(1\),
i.e., \(m \in A\).

It remains to show that \(\End M = A\).
Let \(\phi \in \End M\).
There is a positive integer \(n\) with \(1\cdot n \phi \in A\).
Now \(n \phi - 1 \cdot n \phi\) is an endomorphism of \(M\)
mapping \(1\) to \(0\),
hence it is zero.
Thus \(\phi = 1\phi \in \mathbb{Q}A \cap \End M = A\).
\end{proof}

For the probabilistic version of this construction,
we shall use a consequence of a theorem of Frobenius (see
\citet{frobenius1896beziehungen}) or Chebotarëv's density theorem (see
\citet{chebotarev1923opredelenie, tschebotareff1926bestimmung}) which
is a generalization of Frobenius's theorem.
\begin{lemma}\label{lem:recSumPrimes}
  \label{lem:2}
  For every non-constant (univariate) polynomial \(f \in \Z[x]\)
  the sum of the reciprocals of primes \(p\)
  for which \(f\) has a root modulo \(p\) diverges, i.e.,
  \begin{equation}
    \label{eq:19}
    \sum_{\substack{p \text{ prime} \\
        \text{\(f\) has root mod \(p\)}}}
    \frac{1}{p} = \infty.
  \end{equation}
\end{lemma}

We first specify the distribution according to which we choose the
random elements:

\begin{distribution}
\label{dist:zass}
Let \(A\)  be a ring with a finite-rank, free additive group. For
every prime \(p\) we choose uniformly and independently a non-zero
element \(a_p \in A\) and an integer \(c_p \in [p,2p-1]\).
Whenever \(c_{p} - a_{p}\) is invertible in \(\mathbb{Q} A\),
we choose positive integers \(r_{p}\) and \(d_{p}\) arbitrarily with
\(p^{r_{p}} d_{p} {(c_{p} - a_{p})}^{-1}\) in \(A\)
and \(d_{p}\) relative prime to \(p\).
\end{distribution}

We are ready to prove the probabilistic variant of
Theorem~\ref{thm:Zassenhaus}. 

\begin{theorem}
 \label{thm:ZassenhausProb}
  Let \(A\) be a ring with a finite-rank free additive group and let
  \(M\) be the torsion-free abelian group
  \begin{equation}
    \label{eq:4}
  M \coloneqq \genBy{ A, p^{- r_{p}} (c_p - a_p) A : p \text{
      prime and } \exists {(c_{p} - a_{p})}^{-1} }
  \end{equation}
  with \(c_{p}, a_{p}, r_{p}\) chosen via
  Distribution~\ref{dist:zass}.
  Then \(\End M = A\) almost surely,
  where \(A\) acts on \(M\) by
  multiplication on the right. 
\begin{proof}
With the argumentation in the sketch of proof of
Theorem~\ref{thm:Zassenhaus}, it suffices to prove that
for every pair of non-zero elements \(e\) and \(a\) of \(A\)
there is a prime \(p\) such that \(a_{p} = a\),
the difference \(c_{p} - a_{p}\) is invertible in \(\mathbb{Q} A\),
and
\(p\) divides the order of \({(c_{p} - a_{p})}^{-1} e\) in
\(\mathbb{Q} A / A\).

There are only finitely many \(c\)
for which \(c-a\) is non-invertible,
namely, the roots of the characteristic polynomial of
(left) multiplication by \(a\).
Furthermore \({(c - a)}^{-1} e\) is
a rational function of \(c\) of degree at most \(-1\),
i.e., the coordinates are rational functions
in (any) basis of \(\mathbb{Q} A\).
Now \(p\) divides the order of \({(c_{p} - a_{p})}^{-1} e\) in
\(\mathbb{Q} A / A\) if and only if
there is a coordinate where \(p\) occurs
with negative exponent, e.g., \(p\) divides the denominator
but not the numerator.
We  simplify the coordinates to make
the numerator and denominator relative prime polynomials,
so there are only finitely many primes
such that at any place \(c\), at most these among all primes
divide both the numerator and denominator of a coordinate.
Note that for non-zero coordinates,
the denominator is still non-constant,
as the coordinate has negative degree.

All in all,
there is a non-constant polynomial \(f\)
(the denominator of a coordinate)
with integer coefficients
for which all but finitely many primes \(p\)
and any root \(c\) of \(f\) modulo \(p\),
the order of
\({(c - a)}^{-1} e\) in \(\mathbb{Q} A / A\)
is divisible by \(p\).
For every \(p\) where \(f\) has a root modulo \(p\),
we choose \(c_{p}\) a root with probability at least \(1/p\).
Since these events are independent and
the sum of their probabilities is infinite by \autoref{lem:2},
infinitely many of these events occur almost surely.
Again by independence,
\(a_{p}=a\) almost surely for infinitely many of these \(p\),
finishing the proof.
\end{proof}
\end{theorem}

\section{Concluding remarks and open questions}
So far the proposed method only works for 
constructions up to continuum in size. This is due to the lack of a strong probability theory
beyond \(2^{\aleph_0}\). However we believe that this method is likely to be generalized to 
such cases  as well; a step in this direction is
Lemma~\ref{lem:1}. A potential route to carry over the
probabilistic tools might be to work with a countable model of set
theory; however this is speculation. In particular, the following
questions remain open, where the last one is probably the most
intricate one: 

\begin{enumerate}
\item Generalize to Butler (locally free): There is a well-known
  generalization of Zassenhaus's Theorem~\ref{thm:Zassenhaus} to the
  locally free case by Butler (see \citet{butler1968locally}).
  It turns out that our randomized
  construction does not easily generalize to this case. \medskip

\emph{Can the
  construction be generalized to the locally free case?} \medskip

\item Use randomness in a more involved way: So far randomness has
  been used either to construct algebraically independent elements or
  to ensure that all elements of a countable set have been chosen.
  However randomness has not been directly employed in the construction itself. \medskip

\emph{Can we use randomization in the constructions itself in order to
  obtain simplified or even stronger constructions?}\medskip

\item Generalize beyond \(2^{\aleph_0}\): Probability theory is
  defined on \(\sigma\)-algebras and countability plays a central role
  in the arguments. \medskip

\emph{Are there probabilistic constructions of objects larger than
  \(2^{\aleph_0}\)?} \medskip

\item
  It is independent of ZFC whether
  for \(\aleph_0 < \lambda < 2^{\aleph_0}\),
  the union of \(\lambda\) events of probability \(0\)
  from the continuous uniform distribution
  (e.g., of a random real number from \([0,1]\))
  has again probability zero.\medskip

\emph{Is there a randomized construction for a (natural) statement in
  abelian group theory, so that the actual probability for the theorem
  to hold is \textbf{independent} of ZFC?} \medskip

For example, is there a realization theorem for some (family of)
ring \(A\) so that \(\End(G) = A\) with probability \(1\) in one
universe and probability \(0\) in another?
\end{enumerate}

\section{Acknowledgments}
We are indebted to Brendan Goldsmith for the valuable feedback and
comments as well as pointing us to
related work. We would also like to thank Winfried Bruns for the
helpful discussions on Lemma~\ref{lem:recSumPrimes}.

\bibliographystyle{plainnat}
\bibliography{bibliography}

\end{document}